\documentclass[12pt]{amsart}

\usepackage{amsmath,amsthm,amscd,amsfonts,amssymb,epic,eepic,bbm}
\usepackage[pagebackref,colorlinks=true,linkcolor=blue,citecolor=blue]{hyperref}

\allowdisplaybreaks
\newtheorem{thm}{Theorem}[section]
\newtheorem{cor}[thm]{Corollary}

\newtheorem{lem}[thm]{Lemma}

\newtheorem{prop}[thm]{Proposition}
\theoremstyle{remark}
\newtheorem*{rem}{Remark}

\newcounter{remarkscounter}
\newenvironment{remarks}
{\medskip\noindent{\it
Remarks.}\begin{list}{{\rm(\arabic{remarkscounter})}
}{\usecounter{remarkscounter}

\setlength{\labelsep}{\fill} \setlength{\leftmargin}{0pt}
\setlength{\itemindent}{\fill}
\setlength{\labelwidth}{\fill}\setlength{\topsep}{0pt}
\setlength{\listparindent}{0pt}}} {\end{list}}

\numberwithin{equation}{section}

\newcommand{\A}{\mathbb{A}}
\newcommand{\GL}{\mathrm{GL}}

\newcommand{\ZZ}{\mathbb{Z}}

\newcommand{\Gal}{\mathrm{Gal}}

\newcommand{\QQ}{\mathbb{Q}}

\newcommand{\lto}{\longrightarrow}

\newcommand{\OO}{\mathcal{O}}
\newcommand{\CC}{\mathbb{C}}
\newcommand{\RR}{\mathbb{R}}

\newcommand{\N}{\mathrm{N}}

\newcommand{\quash}[1]{}

\theoremstyle{definition}

\renewcommand{\bar}{\overline}

\numberwithin{equation}{subsection}

\newcommand{\one}{\mathbbm{1}}
\newcommand{\W}{\mathbb{W}}

\begin{document}

\title{A nonabelian trace formula}
\author{Jayce R.~ Getz}
\address{Department of Mathematics\\
Duke University\\
Durham, NC  27708-0320}
\email{jgetz@math.duke.edu}
\author{P.~Edward Herman}
\address{Department of Mathematics\\
University of Chicago\\
Chicago, IL 60637}
\email{peherman@math.uchicago.edu}

\begin{abstract} Let $E/F$ be an extension of number fields with $\mathrm{Gal}(E/F)$ simple and nonabelian.  In a recent paper the first named author suggested an approach to nonsolvable base change and descent of 
automorphic representations of $\mathrm{GL}_2$ along such an extension.
Motivated by this we prove a trace formula whose spectral side is a weighted sum over cuspidal automorphic representations of $\mathrm{GL}_2(\mathbb{A}_E)$ that are isomorphic to their $\mathrm{Gal}(E/F)$-conjugates.
\end{abstract}

\maketitle

\tableofcontents

\section{Introduction} \label{sec-tf}

Let $F$ be a number field and let $H,G$ be reductive $F$-groups.  The Langlands functoriality conjecture predicts that given an $L$-map
\begin{align} \label{L-map}
{}^LH \lto {}^LG
\end{align}
one should have a corresponding transfer of packets of automorphic representations of $H(\A_F)$ to packets of automorphic representations of $G(\A_F)$ compatible with local transfers at every place $v$ of $F$.  If the rank of $G$ is large, in most cases where this conjecture is known the image of the transfer map induced by \eqref{L-map} has a very simple description, namely it consists of $L$-packets fixed by a finite-order automorphism of $G$.  One has a standard technique available to isolate representations that are isomorphic to their twists under an automorphism of
$G$, namely the twisted trace formula, and one can compare this to a trace formula on $H$ to establish the desired transfer; see \cite{AC} for the first set of examples in arbitrary rank and \cite{Arthur} \cite{Mok} for more general examples that represent the state of the art.  The alternate approach via converse theory is different, but the set of $L$-maps treated in arbitrary rank is roughly the same \cite{CKPSS}.

If one could isolate automorphic representations of a group $G$ invariant under a non-cyclic group of automorphisms then one could hope to establish new cases of Langlands functoriality.  In this paper we take the first step in this direction in a very special case.  

\subsection{The trace formula}
It is convenient to first explicitly describe the trace formula we will employ.
Let $F$ be a number field, let $K_\infty \leq \GL_2(F_\infty)$ be
the standard maximal compact subgroup and let
$$
\GL_2(\A_F)^1:=\mathrm{ker}\left( |\cdot| \circ \det:\GL_2(\A_F)\lto \RR_{>0}\right)
$$
where $|\cdot|$ is the usual idelic norm.  
Moreover let $B=TU \leq \GL_2$ be the standard Borel subgroup of upper triangular matrices with 
 $T\leq \GL_2$ the subgroup of diagonal matrices, $U$ the unipotent radical of $B$ and 
\begin{align}
T(\A_F)^1=\left\{\left(\begin{smallmatrix} t_1 & \\ & t_2\end{smallmatrix} \right) \in T(\A_F):|t_1|=|t_2|=1\right\}.
\end{align}

Using these subgroups one 
defines spaces of test functions 
$$
\mathcal{C}(\GL_2(\A_F)^1)
$$ 
as in \S \ref{sec-transf} below. The key property of these spaces of test functions is that, although their elements are not compactly supported functions at any place, the trace formula is still valid for them, as proven in \cite{FL} (and in greater generality in \cite{FLM} and \cite{FL2}).  Explicitly, in \cite{FL} one finds a proof of the following theorem:

\begin{thm}[Arthur-Finis-Lapid-M\"uller-Selberg] \label{thm-tf} For any $f \in \mathcal{C}(\GL_2(\A_F)^1)$ the sum 
\begin{align*}
TF(f):&=\sum_{\gamma/\sim} \mathrm{meas}(\GL_{2\gamma}(F) \backslash \GL_{2\gamma}(\A_F)^1) \int_{\GL_{2\gamma}(\A_F) \backslash \GL_2(\A_F)}f(x^{-1}\gamma x)dx\\
&+\mathrm{meas}(T(F)\backslash T(\A_F)^1)\sum_{t \in  T(F)} \int_{K_\infty \times \GL_2(\widehat{\OO}_E)}\int_{U(\A_F)}f(k^{-1}tuk)\omega(t,u)dudk\\
&+\mathrm{meas}(T(F) \backslash T(\A_F)^1)\sum_{t \in  T(F)}\lambda_{t,S}\int_{K_\infty \times \GL_2(\widehat{\OO}_E)}\int_{U(\A_F)}f(k^{-1}tuk)dudk
\\&-\mathrm{tr}\,R_{\mathrm{res}}(f)+\frac{1}{4\pi}\int_{i\RR}\mathrm{tr}\, M^{-1}(s)M'(s)I(f,s) ds-\frac{1}{4}\mathrm{tr}\,M(0)I(f,0)
\end{align*}
is absolutely convergent and equal to $\mathrm{tr}\,R_{\mathrm{cusp}}(f)$.
\end{thm}
\noindent Here the sum on $\gamma$ is over elliptic conjugacy classes in $\GL_2(F)$ and $R_{\mathrm{res}}(f)$ (resp.~$R_{\mathrm{cusp}}(f)$) denotes the restriction of the usual convolution operator $R(f)$ to the residual (resp.~discrete) spectrum.  Moreover $\GL_{2\gamma}$ denotes the centralizer of $\gamma$ in $\GL_2$.

If $f \in C_c^\infty(\GL_2(\A_F)^1)$ then this formula is a special case of the Arthur-Selberg trace formula.  In general this is \cite[Theorem 1]{FL}, with a caveat: the theorem in loc.~cit. is stated only for $F=\QQ$, but this is only for convenience (see the introduction of \cite{FL}).  We will not explicate the definition of the undefined terms and measures in the expression above; instead we refer the reader to loc.~cit.  From a qualitative point of view, the key point is that the first line of the equality in Theorem \ref{thm-tf} is equal to $\mathrm{tr} \,R_{\mathrm{cusp}}(f)$ up to ``error terms,'' i.e.~the rest of the expression.

\begin{rem} The result is actually better than stated above, for the terms in the trace formula define continuous seminorms on $\mathcal{C}(\GL_2(\A_F)^1)$ (see \cite[Theorem 1]{FL}).  In particular, it is shown in loc.~cit. that the trace norm of $\mathrm{tr}\,R_{\mathrm{cusp}}(f)$ is bounded.
\end{rem}

\subsection{The main theorem}

Let $E/F$ be a Galois extension
with 
$$
\Gal(E/F)=\langle \iota, \tau \rangle
$$ where $|\iota|=2$ and $\tau$ has odd order.  
\begin{rem} Any nonabelian finite simple group can be generated by an involution and an element of odd order by a result of Guralnick and Malle
 (see \cite[Theorem 1.6]{GuMa} and the paragraph after it).
\end{rem}

Let 
\begin{align}
F':={}^{\iota=1}E \quad \textrm{ and } \quad F_0 :={}^{\tau=1}E
\end{align}
denote the fixed fields of $\iota$ and $\tau$, respectively.  We also let
$$
G:=\mathrm{Res}_{\OO_E/\OO_F}\GL_2 \quad \textrm{ and } \quad G_0:=\mathrm{Res}_{\OO_{F_0}/\OO_F}\GL_2.
$$

Let $K_\infty\leq G(F_\infty)$ be the standard maximal compact subgroup and  let
$$
K^\infty:=G(\widehat{\OO}_F) \leq G(\A_F^\infty).
$$   As usual we set $K=K_\infty K^\infty$.  Finally let $S$ be a $\Gal(E/F)$-invariant set of places of $E$ containing all infinite places,  all places where $E/\QQ$ is ramified, and all dyadic places.  We also denote by $S$ the set of places of any subfield $E \geq k \geq F$ dividing $S$.

Let $\Psi \in C_c^\infty(G(\A_F)^1//K)$.
In \S \ref{sec-V} we define certain functions
$$
V:=V_{\Psi} \in \mathcal{C}(G(\A_F)^1) \quad \textrm{ and } \quad b_{E/F_0}V \in \mathcal{C}(G_0(\A_F)^1)
$$ 
such that the following theorem holds:

\begin{thm} \label{thm-main}  
Let $f \in C_c^\infty(G(\A_F)^1)$.  Consider 
\begin{align} \label{star1}
2\sqrt{d}_E\sum_{\substack{\Pi\\ \Pi \cong {}^{\iota} \Pi \cong {}^{\tau} \Pi}}\mathrm{tr}\,\Pi(f*\Psi) \alpha^{\Pi'_S}(W_{0S})L(1,\Pi' \times \Pi'^{\vee} \times \eta^S),
\end{align}
where the sum is over cuspidal automorphic representations of $G(\A_F)^1$ spherical at every place and $\Pi'$ is a cuspidal automorphic representation of $\GL_2(\A_{F'})^1$ whose base change to $E$ is $\Pi$.

 The sum \eqref{star1} is
absolutely convergent.  If $f^{G_0} \in C_c^\infty(G_0(\A_F)^1)$ is a transfer of $f$ then \eqref{star1} is equal to 
\begin{align*}
\frac{1}{[E:F_0]}TF(f^{G_0}*b_{E/F_0}(V)).
\end{align*}
where $TF$ is defined as in Theorem \ref{thm-tf}.
\end{thm}
\noindent
 Here $W_{0S} \in \mathcal{W}(\pi_S,\psi_S)$ is a Whittaker newform (see \S \ref{ssec-Wh-gen}) for $\Pi$ and $\alpha^{\Pi'_S}(W_{0S})$ is defined as in \eqref{a}.
Our conventions regarding transfers of functions are described in \S \ref{sec-transf}.  

The proof of Theorem \ref{thm-main} boils down to carefully constructing the function $V$ and then proving that it lies in the space of test functions to which the trace formula applies.  We will explain this in broad terms in \S \ref{ssec-outline} below.

\subsection{Possible applications}

The authors view Theorem \ref{thm-main} as a nontrivial step towards proving the existence of base change and descent of automorphic representations of $\GL_2$ along the extension $E/F$.  We illustrate this with a typical example.  Assume that $\Gal(E/F)$ and $\Gal(E/F)=\langle \iota,\tau \rangle$ with $|\iota|=2$ and $|\tau|$ odd.  
One does not expect all cuspidal representations $\pi$ of $\GL_2(\A_E)$ that are $\Gal(E/F)$-invariant to descend to $\GL_2(\A_F)$, but Langlands functoriality predicts that one can choose a finite set of characters $\chi:E^{\times} \backslash \A_E^{\times} \to \CC^\times$ depending on $\Gal(E/F)$ such that if $\pi$ is $\Gal(E/F)$-invariant then $\pi \otimes \chi$ descends to $\GL_2(\A_F)$ for some $\chi$ \cite[Proof of Lemma 5.1]{Getz-Ap}.  This additional complication disappears if we assume in addition that $\Gal(E/F)$ is the universal perfect central extension of a finite simple group (compare loc.~cit.).  Thus let us assume that $\Gal(E/F)$ is a universal perfect central extension of a finite simple group. For simplicity we exclude the case where $\Gal(E/F)$ is the universal perfect central extension of the icosahedral group (this case is treated from a slightly different perspective in \cite{Getz-Ap}).

Let 
$$
f \in C_c^\infty(G(\A_F)^1//K).
$$
Consider the two quantities
\begin{align} \label{over-F_0}
TF(f^{G_0}*b_{E/F_0}(V))
\end{align}
and
\begin{align} \label{over-F}
TF(f^{\GL_2}*b_{E/F}(V))
\end{align}
 where $b_{E/?}$, $? \in \{F_0,F\}$ is the base change homomorphism defined on (a subalgebra of) $\mathcal{C}(G(\A_F)^1)$ in \S \ref{sec-transf}, and $f^{G_0} \in C_c^\infty(G_0(\A_F)^1),f^{\GL_{2}} \in C_c^\infty(\GL_2(\A_F)^1)$ are transfers of $f$.  Notice that by Theorem \ref{thm-tf} the quantities \eqref{over-F_0} and \eqref{over-F} are given by (admittedly complicated) absolutely convergent sums and integrals that are geometric in nature.  It is not hard to argue as in \cite{Getz-Ap} that the equality of \eqref{over-F_0} and \eqref{over-F} for all $f$ implies that any $\Gal(E/F)$-invariant cuspidal automorphic representation of $G(\A_E)$ spherical at all places descends to $\GL_2(\A_F)$ (compare \cite[Theorem 1.2]{Getz-Ap}). 

\begin{rem}
 The existence of base changes of cuspidal automorphic representations $\pi$ of $\GL_2(\A_F)^1$ spherical at all places would also follow if we knew somehow that $\mathrm{tr}\,  \pi(b_{E/F}(V)) \neq 0$, but it is not clear how to prove this without knowing the existence of the base change.
  \end{rem}

We emphasize, once again, that \eqref{over-F_0} and \eqref{over-F} are built out of absolutely convergent geometric sums.  This is not easy to come by in this subject.  The primary difficulty in all cases where Langlands' beyond endoscopy proposal \cite{LanglBeyond} has been carried out \cite{Venk}, \cite{Herman}, \cite{PJW} has been obtaining an absolutely convergent sum out of limiting forms of trace formulae, and, to date, it has never been carried out to prove functoriality in a case that was not known by some other method.  As explained in \cite{Sarnak}, this task is expected to become even more difficult in other situations.  One motivation for proving Theorem \ref{thm-main} is that we can contemplate proving that \eqref{over-F_0} is equal to \eqref{over-F}, and this might provide a way to move around the analytic difficulties of the beyond endoscopy proposal in certain cases.

\subsection{Outline} \label{ssec-outline}

In the following section we give our conventions on Haar measures for the paper.  
In \S \ref{sec-Whitt} we record some known results on Whittaker functions, and prove the crucial result Corollary \ref{cor-arch-CSS}, which, combined with the Casselman-Shalika-Shintani formula, allows us to construct Hecke operators $\W_g$ whose Hecke eigenvalues are the values of Whittaker functions\footnote{Analogues of these functions in the function field case  play a key role in the approach of L.~Lafforge to functoriality (see, e.g.~\cite{Lafforge}).} at $g \in G(\A_F)$.  
We then record what we mean by transfers of test functions in our setting in \S \ref{sec-transf}.  

Finally in \S \ref{sec-V} we construct $V$.  The idea is to use the special Hecke operators $\W_g$ to 
build a weight consisting of the period of a cusp form over a unitary group into the trace formula.  This weight isolates representations of $G(\A_F)$ invariant under $\iota$ by work of Jacquet (we actually require the refinements of \cite{FLO}).
These $V$ are then used in \S \ref{ssec-proof} to prove our main result, Theorem \ref{thm-main}.

\section{Haar measures} \label{sec-measures}

Throughout this paper we normalize Haar measures as in the 
following table, which is consistent with the normalizations in \cite{FL}:
\begin{center}
\begin{tabular}{ l l }
\hline
group & normalization\\
\hline
discrete & counting measure\\
$\RR$   & Lebesgue measure\\
$i\RR$ & through $x \mapsto ix$\\
$E_w$, $w$ archimedian & $[E:\RR]$ times the Lebesgue measure on $E_w \cong \RR^{[E:\RR]}$\\
$E_w$, $w$ nonarchimedian & $\mathrm{meas}(\OO_{E_w})=1$\\
$\RR^\times$, $\RR_{>0}$ & $d x^\times=\frac{dx}{|x|}$\\
$E_w^\times$, $w$ archimedian & $d x^\times=\frac{dx}{|x|}$ \\
$E_w^\times$, $w$ nonarchimedian & $\mathrm{meas}(\OO_{E_w}^\times)=1$\\
$\A_E^\times$ & product measure\\
$(\A_E^\times)^1$ & compatible with $\A_E^\times =(\A_E^\times)^1 \times \RR_{>0}$\\
$Z_{\GL_2}(\A)$ & through $t \mapsto \left( \begin{smallmatrix} t &\\ & t \end{smallmatrix} \right)$\\
$U(E_w)$, $U(\A_E)$ & through $x \mapsto \left( \begin{smallmatrix} 1 &x\\ & 1 \end{smallmatrix}\right)$\\
$T(E_w)$, $T(\A_E)$, $T(\A_E)^1$ & through $(t_1,t_2) \mapsto \left( \begin{smallmatrix} t_1 &\\ & t_2 \end{smallmatrix}\right)$\\
$\GL_2(\OO_{E_w})$, $K_w$ ($w$ archimedian) & measure $=1$\\
$\GL_2(E_w)$, $\GL_2(\A_E)$ & compatible with the Iwasawa decomposition\\
$\GL_2(\A_E)^1$ & compatible with $\GL_2(\A_E)=\GL_2(\A_E)^1 \times \RR_{>0}$\\
\hline
\end{tabular}
\end{center}
Here $(\A_E^\times)^1=\mathrm{ker}(|\cdot|)$.  We will also require a Haar measure on $H(\A_F)$ and $H(F) \backslash H(\A_F)$, where $H$ is the group of \S \ref{ssec-H}, but for the purposes of this paper there is no need to specify it; we simply fix one.

\section{Whittaker functions} \label{sec-Whitt}

In this section we collect some generalities on Whittaker functions.  The main results are corollaries \ref{cor-CSS} and \ref{cor-arch-CSS}, which construct Hecke operators whose Hecke eigenvalues are the values of Whittaker functions.

None of the results of this section use anything about the extension $E/F$, so for this section $E$ is an arbitrary number field.  We let $K^\infty:=\GL_2(\widehat{\OO}_E)$ and let $K_\infty \leq \GL_2(E_\infty)$ be the standard maximal compact subgroup.

\subsection{Generalities} \label{ssec-Wh-gen}
Let $\varphi$ be a smooth cuspidal function in  $L^{2}(\GL_2(E) \backslash \GL_2(\A_E)^1)$ and let $\psi:E \backslash \A_E \to \CC^\times$ be the additive character
\begin{align} \label{psi}
\psi:=\psi_\QQ \circ \mathrm{tr}_{\A_E/\A_\QQ}
\end{align}
where $\psi_{\QQ}:=\psi_{\infty} \times \prod_{p}\psi_p$ with
 \begin{align*}
\psi_{\infty}(a):&=e^{-2\pi i a} \quad \textrm{ and } \quad
\psi_{p}(a):=e^{2 \pi i \mathrm{pr}(a)}.
\end{align*}
Here the principal part $\mathrm{pr}(a) \in \ZZ[p^{-1}]$ is any element such that $a \equiv \mathrm{pr}(a) \pmod{\ZZ_p}$.
 We set
\begin{align}
\psi_{m}(x):=\psi(m x)
\end{align}
for $m \in E$.
As usual for each $m \in E$ we have  Whittaker functions
\begin{align} \label{Whitt}
W_{m}^{\varphi}(g):&=\int_{E \backslash \A_E} \varphi\left( \begin{pmatrix} 1 & x \\ 0 & 1 \end{pmatrix} g\right)\psi(-mx)dx.
\end{align}
To ease notation we set
$$
W^{\varphi}:=W_1^{\varphi}.
$$

Since $\varphi$ is cuspidal $W_m^{\varphi}(g)$ is identically zero as a function of
 $g$ if $m=0$.  By taking Fourier expansions we obtain the Whittaker expansion
\begin{align} \label{Wh-exp}
\varphi(g)=\frac{1}{\sqrt{d}_E}\sum_{m \in E^{\times}}W_m^{\varphi}(g).
\end{align}

Assume now that $\varphi$ is in the space $V_\pi$ of a cuspidal automorphic representation $\pi$ of $\GL_2(\A_E)$.  Then  
$W^{\varphi}_m(g)$ is identically zero as a function of $g \in \GL_2(\A_E)$ if and only if $m=0$.  For $m \in E^{\times}$ we obtain nonzero intertwining maps
\begin{align*}
V_{\pi} &\lto \mathcal{W}(\pi,\psi_m)\\
\varphi &\longmapsto W^{\varphi}_m
\end{align*}
to the Whittaker model $\mathcal{W}(\pi,\psi_m)$ of $\pi$ with respect to $\psi_m$.  Here $\mathcal{W}(\pi,\psi_m)$ and its local analogues are defined as in \cite[\S 1.2]{Cogdell}.

We recall that Whittaker functions factor into local Whittaker functions:
\begin{align}
W^{\varphi}_m(g)=\prod_w  W^{\varphi}_{m,w}(g)
\end{align}
where the product is over places $w$ of $E$ and  $W_{m,w}^{\varphi}  \in \mathcal{W}(\pi_w,\psi_{mw})$, the Whittaker model of $\pi_w$ with respect to $\psi_{mw}$.   

We note that for a smooth $\varphi$, a $\beta \in E^{\times}$, and $g \in \GL_2(\A_E)$ one has
\begin{align} \label{move-be}
W_{\beta}^{\varphi}\left(g\right)=W^{\varphi}\left(
\left(\begin{smallmatrix}\beta & \\ & 1 \end{smallmatrix}\right)g \right)
\end{align}
which follows from a change of variables in \eqref{Whitt} together with the left $\GL_2(E)$-invariance of $\varphi$. 

Finally we record a definition used in the statement of Theorem \ref{thm-main} above.  Let 
\begin{align} \label{delta}
\delta \in \A_E^\infty \cap \widehat{\OO}_E
\end{align}
be an idele whose associated ideal is the absolute different $\mathcal{D}_E$ of $E$ over $\QQ$.
Let $\pi_w$ be a generic irreducible admissible representation of $\GL_2(E_w)$.  We say that $W \in \mathcal{W}(\pi_w,\psi_w)$ is a \textbf{Whittaker newform} if the following conditions hold:
\begin{enumerate}
\item It is $K_w$-finite.
\item The identity 
$$
|\delta|_w^{s-1/2}\int_{E^\times_w}W\left(\begin{smallmatrix} m & \\ & 1 \end{smallmatrix} \right)|m|_w^{s-1/2}dm^\times=L(s,\pi_w)
$$
holds and the integral on the left is absolutely convergent for $\mathrm{Re}(s)$ large enough.
\item If $w$ is archimedian, then $W$ has minimial $K_w$-type among all vectors $W'\in \mathcal{W}(\pi_w,\psi_w)$ such that (1) and (2) hold
\item If $w$ is nonarchimedian, then $W$ is a newform.
\end{enumerate}
The factor $|\delta|_w^{s-1/2}$ occurs here because we have chosen our standard character $\psi_w$ to have conductor $\mathcal{D}_{Ew}$.  If $\varphi$ is a smooth function in the space of a cuspidal automorphic representation $\pi$ realized in the cuspidal subspace of  $L^{2}(\GL_2(E) \backslash \GL_2(\A_E)^1)$, then we say that $\varphi$ is a \textbf{Whittaker newform} if $W^{\varphi}$ can be factored as  $W^{\varphi}=\prod_wW_w^{\varphi}$ where $W_w^{\varphi}$ is a Whittaker newform for all $w$.

Whittaker newforms always exist.  For the archimedian case see \cite{Popa}.  In the nonarchimedian case the Whittaker newform is a choice of normalization of a newform (or new vector) in the usual sense (see \cite[Proposition 6.17]{Gelbart} or \cite[\S 2, Summary]{S}).

\subsection{The Casselman-Shalika-Shintani formula}

For the remainder of this section we work in a local setting.  Thus we fix a place $w$ and omit it from notation, writing $E:=E_w$.

Assume for this subsection that $w$ is nonarchimedian.  Let $m \in E^\times \cap \OO_E$.  We let
\begin{align}
\one_m(g):=\one_{\{x \in M_2(\OO_E): \det(x)\OO_E=m\OO_E\}}(g) \in C_c^\infty(\GL_2(E)//K).
\end{align}
Here and below, unless otherwise specified, $\one_X$ is the characteristic function of a set $X$.
  
The following is the Casselman-Shalika formula, due in the case of $\GL_n$ to Shintani:
\begin{thm}\label{thm-CSS} Let $\pi$ be a generic representation of $\GL_2(E)$.  If $W \in \mathcal{W}(\pi,\psi)$ is fixed by $K$ then for $m \in E^\times \cap \OO_E$ one has
\begin{align*} 
(\pi(\one_m)W)\left(\begin{smallmatrix} 1/\delta & \\ & 1\end{smallmatrix} \right)=|m|^{-1}W\left(\begin{smallmatrix} m/\delta & \\ & 1\end{smallmatrix}  \right).
\end{align*}
If $W$ is additionally a newform
$$
\pi(\one_m)W=\mathrm{tr}\,\pi(\one_m)W.
$$
\end{thm}
\noindent Here $\delta$ is as in \eqref{delta}.
\begin{proof}
The second assertion is an immediate consequence of the fact that in each admissible irreducible representation $\pi$ the space of newforms is one-dimensional.  

The identity
$$
(\pi(\one_{m})W)\left( \begin{smallmatrix} 1/\delta & \\ & 1\end{smallmatrix} \right)=|m|^{-1}W\left(\begin{smallmatrix} m/\delta & \\ & 1\end{smallmatrix}  \right)
$$
is well-known.  One reference is \cite[Proposition 4.6]{KL}.  

\end{proof}

Let $B \leq \GL_2$ denote the Borel subgroup of upper-triangular matrices.  
For $h=\left( \begin{smallmatrix} 1 & t \\ & 1 \end{smallmatrix} \right)\left(\begin{smallmatrix} zm/\delta &  \\ & z \end{smallmatrix} \right)k \in B(E)K$ (with $k \in K$) let 
$$
\W_{\left( \begin{smallmatrix} 1 & t \\ & 1 \end{smallmatrix} \right)\left(\begin{smallmatrix} zm/\delta &  \\ & z \end{smallmatrix} \right)k}(g):=\one_{\OO_E}(m)|m|\psi(t)\one_m(z^{-1}g)
$$
where $\psi$ is as in \eqref{psi}. 
We have the following corollary of Theorem \ref{thm-CSS}:
\begin{cor} \label{cor-CSS}
Let $\pi$ be a generic representation of $\GL_2(E)$.  If $W \in \mathcal{W}(\pi,\psi)$ is a newform and $g \in B(E)K$ then
$$
\pi(\W_g)W= W(g)W.
$$ 
\end{cor}
\begin{proof}
Recall that if $W \in \mathcal{W}(\pi,\psi)$ is a Whittaker newform then $W\left(\begin{smallmatrix} m/\delta & \\ & 1 \end{smallmatrix} \right) \neq 0$ only if $m \in \OO_E$ and $W\left(\begin{smallmatrix} m/\delta & \\ & 1 \end{smallmatrix} \right)=1$.  One reference for this fact is \cite[\S 2]{S}, bearing in mind that in loc.~cit.~the character $\psi$ is normalized to have conductor $\OO_E$, whereas we normalized our $\psi$ to have conductor $\mathcal{D}_E$.  In view of this the corollary is an immediate consequence of Theorem \ref{thm-CSS}.  
\end{proof}

\subsection{A replacement for Theorem \ref{thm-CSS} in the archimedian case}
\label{ssec-arch-CSS}

Assume that $E=E_w$ is an archimedian local field.
In this section we construct test functions in 
\begin{align}
\mathcal{C}(\GL_2(E)),
\end{align}
which is defined to be the Frechet space of smooth functions $f$ on $\GL_2(E)$ such that for all 
\begin{align}
X,Y \in U(\mathrm{Res}_{E/\RR}\mathfrak{gl}_{2E} \otimes_{\RR} \CC)
\end{align}
one has
\begin{align}
||X*f*Y||_{1}:=\int_{\GL_2(E)}|X*f*Y|dg<\infty.
\end{align}
The implicit constant in the choice of Haar measure is irrelevant for our purposes.  

In order to construct these functions we recall the explicit description of the standard model of an induced irreducible admissible representation of $\GL_2(E)$.
The basic reference for the constructions is \cite[\S 5-6]{JL}.   There is a brief summary in \cite[\S 2.2]{Zh}.

Let $(\pi,V_{\pi})$ be an irreducible representation of $\GL_2(E)$ with central character $\chi$. Consider the space $\mathcal{B}(\mu \chi,\mu^{-1})$ of $K$-finite functions $\varphi \in C^\infty(\GL_2(E))$ satisfying
\begin{align} \label{ps}
\varphi\left(z\left(\begin{smallmatrix} y & * \\ & 1\end{smallmatrix}\right)g \right)= \chi(z)\mu\chi(y)|y|^{1/2}\varphi(g) \textrm{ for }(y,g) \in E^{\times} \times \GL_2(E)
\end{align}
where $\mu:E^\times \to \CC^\times$ is a quasicharacter.  The group $\GL_2(E)$ acts on $\mathcal{B}(\mu \chi,\mu^{-1})$ via right translations, and upon replacing $\mu$ with $\chi^{-1}\mu^{-1}$ if necessary we can and do assume that  $\pi$ is equivalent to a subrepresentation (not merely a subquotient) of $\mathcal{B}(\mu\chi,\mu^{-1})$.  
The space $\mathcal{B}(\mu \chi, \mu^{-1})$ consists of vectors of the form
\begin{align} \label{varphiphi}
\varphi_{\Phi}(g):=\mu\chi(\det g)|\det g|^{1/2} \int_{E^{\times}}\Phi((0,t)g) \mu^2 \chi(t)|t| dt^{\times} 
\end{align}
for $\Phi$ in the Schwarz space $\mathcal{S}(E \times E)$.

For $g \in M_2(E)$ let
\begin{align} \label{Fs}
\mathcal{F}_{P}(g):=P(g)e^{-[E:\RR]\mathrm{tr}\,g^*g}|\det g|^{\tfrac{1}{2}}
\end{align}
where $P(g)$ is a polynomial in the entries of $g$ and their complex conjugates.   It is clear that $\mathcal{F}_{P}$ restricts to a smooth function on $\GL_2(E)$. 

For real numbers $r\in \RR_{>0}$ let $\delta_r$ denote the distribution on $\GL_2(E)$ given on test functions $f \in C_c^\infty(\GL_2(E))$ by
$$
\delta_r(f):=\int_{\GL_2(E)^1}f(r g^1)dg^1
$$
where
$$
\GL_2(E)^1:=\mathrm{ker}(|\cdot| \circ \det:\GL_2(E) \to \RR_{>0})
$$
and $dg^1$ is chosen to be compatible with the decomposition 
$$
\GL_2(E)=\GL_2(E)^1 \times \RR_{>0}
$$
and the choices of measures in \S \ref{sec-measures}.
 Less formally, $\delta_r$ is the Dirac distribution associated to the $\GL_2(E)^1$-homogeneous submanifold
$$
\{g \in \GL_2(E):|\det g|=r^{2[E:\RR]}\}
$$
of $\GL_2(E)$.

Assume that the representation 
$$
K \times K \lto \mathrm{Aut}\langle x \mapsto \mathcal{F}_{P}(k_1^{-1}xk_2):k_1,k_2 \in K\rangle
$$
is finite-dimensional and equivalent to an irreducible representation of the form $\rho \otimes \rho^{\vee}$ of $K \times K$, where $\rho$ is an irreducible representation of $K$.

\begin{lem} \label{lem-op}
Assume that $\pi$ is an admissible representation of $\GL_2(E)$.  One has a well-defined bounded operator 
$$
\pi(\mathcal{F}_{P}\delta_{r})
$$
on $\pi$.
Let $\varphi$ be a vector in $\pi$, viewed as a subrepresentation of $\mathcal{B}(\mu \chi,\mu^{-1})$.  If $\varphi$ does not lie in the $\rho$-isotypic subspace of $\pi$, then $\pi(\mathcal{F}_{P}\delta_r)\varphi=0$.   If $\varphi$ lies in the $\rho$-isotypic subspace of $\pi$, then $\pi(\mathcal{F}_{P}\delta_r)$ acts via the scalar 
$$
\int_{E^{\times}  \times E}\mathcal{F}_{P}\left(\begin{matrix} \frac{r}{|y|^{(2[E:\QQ])^{-1}}}y & x \\ & \frac{r}{|y|^{1/(2[E:\RR])}}\end{matrix} \right)\chi(r/|y|^{1/(2[E:\RR])})\mu\chi(y)\frac{dy^{\times}}{|r|}dx
$$
on $\varphi$.  Moreover, the integral in this expression is absolutely convergent.  
\end{lem}

\begin{proof} To prove that the operator is bounded we proceed as in \cite[\S 1.1]{GJ}.
Let $\mathcal{G}_1 \in \mathcal{S}(M_2(E))$ be a nonnegative Schwartz function.  Then
there are nonnegative Schwartz functions  $\mathcal{G}_2 \in \mathcal{S}(M_2(E))$ and $\mathcal{G}_3 \in \mathcal{S}(E)$ such that
\begin{align} \label{schw-comp}
&\int_{\GL_2(E)^1}\mathcal{G}_1(r g^1)dg^1\\
  \nonumber &=\int_{E^\times  \times E \times K}\mathcal{G}_1\left(\frac{r}{|y|^{1/(2[E:\RR])}}\begin{pmatrix} y & x \\ & 1\end{pmatrix}k\right)\frac{dy^\times}{|y|}dxdk
\\ \nonumber &=\int_{E^\times  \times E}\mathcal{G}_2\left(\frac{r}{|y|^{1/(2[E:\RR])}}\begin{pmatrix} y & x \\ & 1\end{pmatrix}\right)\frac{dy^\times}{|y|}dx\\
\nonumber &=\int_{E^\times  \times E}\mathcal{G}_2\begin{pmatrix} \frac{r}{|y|^{1/(2[E:\RR])}}y & x \\ &\frac{r}{|y|^{1/(2[E:\RR])}} \end{pmatrix}\frac{dy^\times}{|y|^{1/2}r^{[E:\RR]}}dx\\
\nonumber &=\int_{E^\times}\mathcal{G}_3\left(
\frac{r}{|y|^{1/(2[E:\RR])}}y ,\frac{r}{|y|^{1/(2[E:\RR])}}\right)\frac{dy^\times}{|y|^{1/2}r^{[E:\RR]}} <\infty.
\end{align}
This implies that $\pi(\mathcal{F}_P\delta_r)$ is a well-defined bounded operator.

 If $\varphi$ does not lie in the $\rho$-isotypic subspace of $\pi$, then the assertion of the lemma is clear. Suppose therefore that $\varphi \neq 0$ does lie in the $\rho$-isotypic subspace (which implies in particular that this subspace is nonempty).  Then our assumptions on the $K$-type of $\mathcal{F}_{P}$ and Schur's lemma imply that the operator $\pi(\mathcal{F}_{P}\delta_r)$ acts via a scalar on 
the $\rho$-isotypic subspace.  We must compute this scalar.  

Choose $\Phi \in \mathcal{S}(E \times E)$ such that $\varphi_{\Phi}$ is in the $\rho$-isotypic subspace of $\pi$ and 
\begin{align} \label{nonzero}
\varphi_{\Phi}\left(\begin{smallmatrix} 1 & \\ & 1 \end{smallmatrix} \right) \neq 0.
\end{align}
That this is always possible can be seen by realizing that the integral defining $\varphi_{\Phi}$ is a Tate integral, compare \cite[Lemma 5.13.1 and Lemma 6.3.1]{JL}.

Using the Iwasawa decomposition we compute
\begin{align*}
\left(\pi(\mathcal{F}_{P}\delta_r)\varphi_{\Phi}\right)\left( \begin{smallmatrix} 1 & \\ & 1\end{smallmatrix} \right)&=\int_{\GL_2(E)^1}\mathcal{F}_{P}(r g^1)\varphi(rg^1)dg^1\\
&=\int_{E^\times  \times E \times K}\mathcal{F}_{P}\left(\frac{r}{|y|^{1/(2[E:\RR])}}\left(\begin{matrix} y & x \\ & 1\end{matrix} \right)k\right)\varphi_{\Phi}\left(\frac{r}{|y|^{1/(2[E:\RR])}}\left(\begin{matrix} y & x \\ & 1\end{matrix} \right)k\right)\frac{dy^{\times}}{|y|}dxdk
\\&=\int_{E^{\times} \times E }\mathcal{F}_{P}\left(\frac{r}{|y|^{1/(2[E:\RR])}}\left(\begin{matrix} y & x \\ & 1\end{matrix} \right)\right)\varphi_{\Phi}\left(\frac{r}{|y|^{1/(2[E:\RR])}}\left(\begin{matrix} y & x \\ & 1\end{matrix} \right)\right)\frac{dy^{\times}}{|y|}dx
\end{align*}
(see \S \ref{sec-measures} for normalizations of measures).

Substituting the definition \eqref{varphiphi} of $\varphi_{\Phi}$ we obtain
\begin{align*}
&\int_{E^\times}
\Phi(0,t)
\mu^2\chi(t)|t|dt^{\times}\int_{E^{\times}  \times E }\mathcal{F}_{P}\left(\frac{r}{|y|^{1/(2[E:\RR])}}\left(\begin{matrix} y & x \\ & 1\end{matrix} \right)\right)\mu\chi(y)\chi(r/|y|^{1/(2[E:\RR])})\frac{dy^{\times}}{|y|^{1/2}}dx\\
&=\varphi_{\Phi}\left( \begin{smallmatrix}1 & \\ & 1 \end{smallmatrix}\right)
\int_{E^{\times}  \times E}\mathcal{F}_{P}\left(\begin{matrix} \frac{r}{|y|^{1/(2[E:\RR])}}y & x \\ & \frac{r}{|y|^{1/(2[E:\RR])}}\end{matrix} \right)\chi(r/|y|^{1/(2[E:\RR])})\mu\chi(y)\frac{dy^{\times}}{|r|}dx.
\end{align*}
By \eqref{nonzero} we conclude that the scalar by which $\pi(\mathcal{F}_{P}\delta_r)$ acts on the $\rho$-isotypic subspace is
\begin{align*} 
&
\int_{E^{\times}  \times E}\mathcal{F}_{P}\left(\begin{matrix} \frac{r}{|y|^{1/(2[E:\RR])}}y & x \\ & \frac{r}{|y|^{1/(2[E:\RR])}}\end{matrix} \right)\chi(r/|y|^{1/(2[E:\RR])})\mu\chi(y)\frac{dy^{\times}}{|r|}dx.
\end{align*}
\end{proof}

For the remainder of this subsection assume that $\pi$ is spherical.  
For each $h=\left(\begin{smallmatrix} 1 & t \\ & 1 \end{smallmatrix} \right)\left(\begin{smallmatrix} zm &  \\ & z \end{smallmatrix} \right)k \in \GL_2(E)$ with $k \in K$  define a distribution $\W_{h}$ on $\GL_2(E)$ by setting
\begin{align}
\W_{\left(\begin{smallmatrix} 1 & t \\ & 1 \end{smallmatrix} \right)\left(\begin{smallmatrix} zm &  \\ & z \end{smallmatrix} \right)k}(g)
\end{align}
equal to 
\begin{align*}
\frac{\psi(t)|m|^{1/2}}{2}\mathcal{F}_{1}\delta_{|m|^{1/(2[E:\RR])}}(z^{-1}g).
\end{align*}

\begin{lem} \label{lem-in-C}
For any $\Psi \in C_c^\infty(\GL_2(E))$ one has $\Psi*\W_g \in \mathcal{C}(\GL_2(E))$.  In fact, for any $A \in \RR_{>0}$ and
$$
X,Y \in U(\mathrm{Res}_{E/\RR}\mathfrak{gl}_{2E} \otimes_{\RR} \CC)
$$
one has
$$
||X*\Psi*\W_{\left(\begin{smallmatrix} m & \\ & 1\end{smallmatrix} \right)}*Y||_1 \ll_{\Psi,A,X,Y} |m|^{-A} 
$$
for $|m|>1$ and there exists a $B>0$ such that 
$$
||X*\Psi*\W_{\left(\begin{smallmatrix} m & \\ & 1\end{smallmatrix} \right)}*Y||_1 \ll_{\Psi,B,X,Y} |m|^{-B} 
$$
for $|m| \leq 1$.
\end{lem}

\begin{proof}

It clearly suffices to prove the second statement.

As explained in \cite[\S 1.8, p.~115-116]{GJ}, we have 
\begin{align}
X*\mathcal{F}_{1}*Y=\mathcal{F}_{P}
\end{align}
for some $P$. Moreover any operator in the subalgebra
$$
U(\mathrm{Res}_{E/\RR}\mathfrak{sl}_{2E} \otimes_{\RR} \CC) \leq 
U(\mathrm{Res}_{E/\RR}\mathfrak{gl}_{2E} \otimes_{\RR} \CC)
$$
is insensitive to the difference between $\mathcal{F}_{P}$ and $\mathcal{F}_{P}\delta_{|m|}$, and the algebra on the right is generated (as an algebra) by these elements together with central operators corresponding to the Lie algebra of $Z_{\mathrm{Res}_{E/\RR}\GL_2}$.  These operators can be absorbed into $\Psi$.  
It therefore suffices to show that for $\Psi \in C_c^\infty(\GL_2(E))$ and polynomials $P$ as above we have
\begin{align} \label{to-show}
\left|\left|\Psi*\mathcal{F}_{P}(g)\delta_{r}\right|\right|_1 &\ll_{\Psi,T,A,P} r^{-A} \textrm{ for }r>1\\
\left|\left|\Psi*\mathcal{F}_{P}(g)\delta_{r}\right|\right|_1 &\ll_{\Psi,T,B,P} r^{-B} \textrm{ for }r\leq 1. \nonumber
\end{align}
for any $A>0$ and some $B>0$.
But 
\begin{align*}
\left|\left|\Psi*\mathcal{F}_{P}(g)\delta_{r}\right|\right|_1
&=\int_{\GL_2(E)} \left|\int_{\GL_2(E)^1} \Psi(x(r g^1)^{-1})\mathcal{F}_P(rg^1)dg^1\right|dx\\
&\leq \int_{\GL_2(E)} \int_{\GL_2(E)^1} \left|\Psi(x(r g^1)^{-1})\mathcal{F}_P(rg^1)\right|dg^1dx\\
&=\int_{\GL_2(E)} \int_{\GL_2(E)^1} \left|\Psi(x)\mathcal{F}_P(rg^1)\right|dg^1dx\\
&\ll_{\Psi}\int_{\GL_2(E)^1} \left|\mathcal{F}_P(rg^1)\right|dg^1.
\end{align*}
Applying \eqref{schw-comp} the assertion \eqref{to-show} is clear.
\end{proof}

The reason for the notation $\W_g$ is apparent given the following Corollary:
\begin{cor} \label{cor-arch-CSS}
Let $\pi$ be unitary, and let $W \in \mathcal{W}(\pi,\psi)$ be the spherical Whittaker newform.  Then $\pi(\W_g)$ annihilates the orthogonal complement of the line spanned by $W$, and acts via the scalar $W(g)$ on the line spanned by $W$.  
\end{cor} 

\begin{proof}
It is enough to verify the corollary in the special case where $g=\left(\begin{smallmatrix} m & \\ & 1 \end{smallmatrix} \right)$.  By Lemma \ref{lem-op} it suffices to show that 
\begin{align} \label{eq}
\mathrm{tr}\,\pi\left(\W_{\left(\begin{smallmatrix} r & \\ & 1 \end{smallmatrix} \right)}\right)=W\left(\begin{smallmatrix} r & \\ & 1 \end{smallmatrix} \right)
\end{align}
for $r \in \RR_{>0}$.
Both sides of this expression are smooth functions of $r$.  It is well-known that the right hand side is rapidly decreasing as $r \to \infty$ and of order at most $r^{-B}$ for some $B>0$ as $r \to 0$ and the same is true of the left hand side by Lemma \ref{lem-in-C}.  Thus to prove that the two are equal it is enough to verify that their Mellin transforms are equal.  
Recall (see \S \ref{sec-measures}) that our choice of Haar measure on $E$ is the Lebesgue measure if $E $ is real and twice the Lebesgue measure (on $\RR^2)$ if $E$ is complex.  Thus, essentially by definition, for $\mathrm{Re}(s)$ sufficiently large one has
$$
2^{[E:\RR]} \pi^{[E:\RR]-1}\int_{\RR_{>0}} W\left(\begin{smallmatrix}r & \\ & 1 \end{smallmatrix} \right)r^{[E:\RR](s-\tfrac{1}{2})}\frac{dr}{r}=L(s,\pi).
$$
On the other hand, by Lemma \ref{lem-op} for $\mathrm{Re}(s)$ sufficiently large
\begin{align*}
&2^{[E:\RR]}\pi^{[E:\RR]-1}\int_{\RR>0}\mathrm{tr}\,\pi\left(\W_{\left(\begin{smallmatrix} r & \\ & 1 \end{smallmatrix} \right)}\right)r^{[E:\RR](s-\tfrac{1}{2})}dr^\times\\
&=2^{[E:\RR]-1}\pi^{[E:\RR]-1}\int_{\RR_{>0}}\int_{E^{\times}  \times E}\mathcal{F}_{1}\left(\begin{matrix}\left( \frac{r}{|y|^{1/[E:\RR]}}\right)^{1/2}y & x \\ &\left( \frac{r}{|y|^{1/[E:\RR]}}\right)^{1/2}\end{matrix} \right)\\& \times \chi\left(\left( \frac{r}{|y|^{1/[E:\RR]}}\right)^{1/2}\right)\mu\chi(y)dy^{\times}r^{[E:\RR](s-\tfrac{1}{2})}dr^\times dx
\\&=2^{[E:\RR]-1}\pi^{[E:\RR]-1}\int_{\RR_{>0}}\int_{E^{\times}  \times E}\mathcal{F}_{1}\left(\begin{matrix}\sqrt{r}y & x \\ & \sqrt{r}\end{matrix} \right)\chi\left(\sqrt{r}\right)\mu\chi(y)|yr|^{s-\tfrac{1}{2}}dy^{\times}dr^\times dx\\
&=\int_{E^\times}\int_{E^{\times}  \times E}\mathcal{F}_{1}\left(\begin{matrix} zy & x \\ & z\end{matrix} \right)\chi(z)\mu\chi(y)|yz^2|^{s-\tfrac{1}{2}}dy^{\times}dz^\times dx\\
&=\int_{E^\times}\int_E \mathcal{F}_{1}\left(\begin{smallmatrix} y & x \\ & z\end{smallmatrix}\right)dx |yz|^{s-1/2}\mu\chi(y)\mu^{-1}(z)dy^{\times}dz^\times\\
&=\int_{E^\times \times E^{\times}}\int_ E e^{-[E:\RR](x\bar{x}+y \bar{y}+z\bar{z})}dx |yz|^{s}\mu\chi(y)\mu^{-1}(z)dy^{\times}dz^\times\\
&=\int_{E^\times \times E^{\times}}e^{-[E:\RR](y \bar{y}+z\bar{z})}  |yz|^{s}\mu\chi(y)\mu^{-1}(z)dy^{\times}dz^\times\\
&=L(s,\mu\chi)L(s,\mu^{-1})=L(s,\pi).
\end{align*}

\end{proof}

\section{Test functions and transfer} \label{sec-transf}

For this section we let $E/F$ be a finite extension of number fields (not necessarily Galois) that is everywhere unramified and split at all infinite places.   Let $G:=\mathrm{Res}_{\OO_E/\OO_F}\GL_2$, let  $K_\infty\leq G(F_\infty)$ be the standard maximal compact subgroup and let $K^\infty \leq G(\A_F^\infty)$ be a compact open subgroup.

Consider the Frechet space 
\begin{align}
\mathcal{C}(G(\A_F)^1):=\mathcal{C}(G(\A_F)^1,K^\infty).
\end{align}
It consists of those smooth functions $f$ on $G(\A_F)$ that are right invariant under $K^\infty$ and satisfy 
\begin{align} \label{L1-bound}
|| X* f *Y ||_{L^1(G(\A_F)^1)}< \infty
\end{align}
for all $X,Y \in U(\mathfrak{g})$, where 
$$
\mathfrak{g}:=\mathrm{Lie}(G(F \otimes_{\QQ}\RR)) \otimes_{\RR}\CC.
$$
We also have the analogous Frechet space $\mathcal{C}(G(F_\infty))$ at infinity; it consists of elements of $C^\infty(G(F_\infty))$ such that \eqref{L1-bound} is valid for all $X,Y$ with $L^1(G(\A_F)^1)$ replaced by $L^1( G(F_\infty))$.  Note that this is not the same as the Harish-Chandra Schwartz space, whose elements may not be integrable.

The spaces
$$
\mathcal{C}(\GL_2(\A_F)^1):=\mathcal{C}(\GL_2(\A_F), K^\infty \cap \GL_2(\A_F^\infty)) \quad \textrm{ and } \quad \mathcal{C}(\GL_2(F_\infty))
$$ 
are defined similarly.
We note that these spaces are algebras under convolution.  Moreover, if $\Pi$ is a cuspidal automorphic representation of $G(\A_F)^1$, then 
$$
\mathrm{tr}\,\Pi:\mathcal{C}(G(\A_F)^1) \lto \CC
$$
is a continuous linear functional (this is part of the main result of \cite{FL}, for example, see Theorem \ref{thm-tf} above).  The assumption that $\Pi$ is cuspidal can be weakened somewhat, but it is not true for an arbitrary admissible representation.

Let $v$ be a place of 
$F$.  We say that $\Psi \in C_c^\infty(\GL_2(F_v))$ is a \textbf{transfer} of $f \in C_c^\infty(G(F_v))$ if 
\begin{align} \label{trace-id}
\mathrm{tr}\, \pi(\Psi)=\mathrm{tr}\, \Pi(f)
\end{align}
is valid for all irreducible admissible representations $\pi$ of $\GL_2(F_v)$ and base change $\Pi$ to $G(F_v)$ (the notion of base change in this context is recalled in \cite[\S 1]{Getz-Ap}).  

\begin{remarks}

\item This notion does not coincide with the usual notion of transfer in the theory of cyclic base change.
\item It would be interesting to define transfers of arbitrary functions in $\mathcal{C}(G(\A_F)^1)$, but we will not attempt this.  Instead we will make do with something weaker (see Lemma \ref{lem-C-trans} below).

\end{remarks}

In this paper we need to make explicit our choices of transfers. 
For the remainder of this section let
$$
K^\infty:=G(\widehat{\OO}_F)
$$
and let $K_F:=K \cap \GL_{2}(\A_F)$.
We will define a continuous algebra homomorphism
\begin{align*}
b_{E/F}:\mathcal{C}(G(\A_F)//K) \lto \mathcal{C}(\GL_2(\A_F)//K_F)
\end{align*}
(see Lemma \ref{lem-C-trans} below).

First recall that for finite places $v$ of $F$ the base change maps
$$
b_{E/F}:C_c^\infty(G(F_v)//K_v) \lto C_c^\infty(\GL_2(F_v)//\GL_2(\OO_{F_v}))
$$
provide examples of transfers in the sense that $b_{E/F}(f)$ is a transfer of $f$ (compare \cite[\S 3.2]{Getz-Ap}).

Second, recall that $E/F$ is assumed to be split at all infinite places. We can therefore define a map
$$
b_{E/F}:\mathcal{C}(G(F_\infty)//K_\infty) \lto \mathcal{C}(\GL_2(F_\infty)//K_{F\infty})
$$
by choosing an identification
\begin{align} \label{G-isom}
G(F_\infty) = \GL_2(F_\infty)^{[E:F]}
\end{align}
taking $K_\infty$ to $K_{F\infty}^{[E:F]}$
and letting
\begin{align} \label{bc-inf}
b_{E/F}(f)(g):=\int_{\GL_2(F_\infty)}\cdots\int_{\GL_2(F_\infty)}f(h_1,\dots,h_{[E:F]-1},h_{[E:F]-1}^{-1}h_{[E:F]-2}^{-1}\dots h_1^{-1}g)dh_1\dots dh_{[E:F]-1}.
\end{align}
  On the dense subspace
$$
\mathcal{C}(\GL_2(F_\infty)//K_\infty)^{\otimes[E:F]}\leq \mathcal{C}(\GL_2(F_\infty)//K_\infty)
$$
this map just coincides with the $[E:F]$-fold convolution, and therefore \eqref{bc-inf} is continuous.
 There is ambiguity in the choice of isomorphism \eqref{G-isom}, but since spherical Hecke algebras are commutative this does not affect the definition of the morphism $b_{E/F}$.

\begin{lem} \label{lem-inf-cont}  Let $\pi$ be a cuspidal automorphic representation of $\GL_2(\A_F)$ and base change $\pi_E$ to $G(\A_F)$.  If
 $f \in \mathcal{C}^1(G(F_\infty)//K_\infty)$, then 
$\mathrm{tr}\,\pi_\infty(b_{E/F}(f))=\mathrm{tr}\,\pi_{E\infty}(f)$.
\end{lem}

\begin{proof}
Under the isomorphism \eqref{G-isom} the representation $\pi_{E\infty}$ corresponds to the $[E:F]$-fold tensor product of $\pi_\infty$.  Thus
if $f$ lies in the dense subset $C^\infty_c(\GL_2(F_\infty)//K_{F\infty})^{\otimes [E:F]}$ then the identity is trivial.  The general case follows from the continuity of the distributions $\mathrm{tr}\,\pi_E$ on $\mathcal{C}(G(\A_F)^1)$ and $\mathrm{tr}\,\pi$ on $\mathcal{C}(\GL_2(\A_F)^1)$.

\end{proof}

The considerations above yield continuous algebra morphisms
\begin{align} \label{star}
b_{E/F}:\mathcal{C}(G(F_\infty)//K_{\infty}) \otimes C_c^\infty(G(\A_F^\infty)//K^\infty)  \lto \mathcal{C}(\GL_2(F_\infty)//K_{F\infty}) \otimes C_c^\infty(\GL_2(\A_F^\infty)//\GL_2(\widehat{\OO}_F)).
\end{align}
\begin{lem} \label{lem-C-trans} The map \eqref{star} extends to a continuous algeba morphism
$$
b_{E/F}:\mathcal{C}(G(\A_F)^1//K^\infty) \lto \mathcal{C}(\GL_2(\A_F)^1//\GL_2(\widehat{\OO}_F))
$$
Moreover, if $\pi$ is a cuspidal automorphic representation of $\GL_2(\A_F)$ with cuspidal base change $\pi_E$ to $\GL_2(\A_E)$ and $f \in \mathcal{C}(G(\A_F)^1//K^\infty)$ then 
$$
\mathrm{tr}\,\pi_E(f)=\mathrm{tr}\,\pi(b_{E/F}(f)).
$$
\end{lem}

\begin{proof} By Lemma \ref{lem-inf-cont} $b_{E/F}$ is continuous and linear on the dense subspace 
$$
\mathcal{C}(G(F_\infty)//K_\infty) \otimes C_c^\infty(G(\A_F^\infty)//K^\infty) 
$$ 
of $\mathcal{C}(G(\A_F)^1//K)$.  Since $\mathcal{C}(G(\A_F)//K)$ is a Frechet space, $b_{E/F}$ therefore extends uniquely to a continuous algebra morphism.  To deduce the final assertion of the lemma we invoke the continuity of the distributions $\mathrm{tr}\,\pi_E$ on $\mathcal{C}(G(\A_F)^1)$ and $\mathrm{tr}\,\pi$ on $\mathcal{C}(\GL_2(\A_F)^1)$ together with Lemma \ref{lem-inf-cont}.
\end{proof}

\section{The construction of $V$}

\label{sec-V}

In this section we construct $V \in \mathcal{C}(G(\A_F)^1)$ and use it to prove our main theorem, Theorem \ref{thm-main}.  We assume that $E/F$ is a Galois extension of number fields that is everywhere unramified and split at all infinite places.  We also assume that
$$
\Gal(E/F)=\langle\iota,\tau \rangle
$$
where $\tau$ has odd order.
We let $F'$ be the fixed field of $\iota$ and $F_0$ the fixed field of $\tau$.  

\subsection{The subgroup $H$} \label{ssec-H}

Let $U_2$ be the unitary group over $\OO_{F'}$ whose points in an $\OO_{F'}$-algebra $R$ are given by 
$$
U_2(R):=\{g \in \GL_2(\OO_E \otimes_{\OO_{F'}}R): g=\left(\begin{smallmatrix}  & 1 \\ 1 &\end{smallmatrix} \right){}^{\iota}g^{-t}\left(\begin{smallmatrix}  & 1 \\ 1 &\end{smallmatrix} \right)\}.
$$
Moreover let
\begin{align} \label{H}
H:=\mathrm{Res}_{\OO_{F'}/\OO_F}U_2 \leq G.
\end{align}
Recall that a measurable subset $\mathfrak{F} \leq H(\A_F)$ is a \textbf{fundamental domain} if $H(F)\mathfrak{F}=H(\A_F)$ and $\gamma \mathfrak{F} \cap \mathfrak{F}=\emptyset$  for all $\gamma \in H(F)$.

Let $B_H \leq H$ be the Borel subgroup of upper-triangular matrices.

\begin{lem} \label{lem-F} There exists a fundamental domain $\mathfrak{F} \subseteq H(\A_F)$ contained in a set of the form
$$
\{ \Omega \left( \begin{smallmatrix} t & \\ &t^{-1} \end{smallmatrix}\right)K \cap H(\A_F):t>T\}
$$
for some $T >0$ and some compact subset $\Omega \leq B_H(\A_F)$.
\end{lem}

\noindent Here we regard the matrices $\left( \begin{smallmatrix} t & \\ &t^{-1} \end{smallmatrix}\right)$, $t \in \RR^\times$ as elements of $B_H(F_\infty)$ via the diagonal embedding.

\begin{proof} This follows from \cite[Theorem 5.3]{Borel} together with the Iwasawa decomposition.
\end{proof}

Let $\mathfrak{F}$ be a fundamental domain as in Lemma \ref{lem-F} and let $\Psi \in C_c^\infty(G(\A_F)^1//K)$.
We define
\begin{align} \label{defn}
V:=V_\Psi:=\int_{\mathfrak{F}} \sum_{y \in E^\times}\Psi*\W_{\left(\begin{smallmatrix} y& \\ & 1 \end{smallmatrix} \right)h}dh.
\end{align}

\begin{rem} The operator
$$
\sum_{y \in E^\times}\Psi*\W_{\left(\begin{smallmatrix} y& \\ & 1 \end{smallmatrix} \right)h},
$$
as a function of $h$, is invariant under the subgroup of $B_H(F)$ of matrices of the form $\left(\begin{smallmatrix} a & b \\ & 1 \end{smallmatrix} \right)$ on the left and $K$ on the right.  It is not invariant under all of $H(F)$ on the left, however, hence the need for a choice of fundamental domain.  We have chosen our fundamental domain in order to guarantee the convergence of $V$.  
\end{rem}

\begin{prop}
One has that $V \in \mathcal{C}(G(\A_F)^1)$. 
\end{prop}

\begin{proof} Let $X, Y \in U(\mathfrak{g})$.  Then by the Iwasawa decomposition
\begin{align*}
&||X*V*Y||_{L^1(G(\A_F)^1)}\\&=\int_{G(\A_F)^1}\left|X*\int_{\mathfrak{F}} \sum_{y \in E^\times}\Psi*\W_{\left(\begin{smallmatrix} y& \\ & 1 \end{smallmatrix} \right)h}dh*Y(g)\right|dg\\
&=\int_{G(\A_F)^1}\left|X*\int_{\left(\begin{smallmatrix} ta& t^{-1}ab\\ & t^{-1}{}^{\iota}a^{-1} \end{smallmatrix} \right)k_H \in \mathfrak{F}} \sum_{y \in E^\times}\Psi*\W_{\left(\begin{smallmatrix} yta& yt^{-1}ab\\ & t^{-1}{}^{\iota}a^{-1} \end{smallmatrix} \right)k_H}dt^\times da^\times dbdk_H*Y(g)\right|dg
\end{align*}
where $t \in \RR_{>0}$ is embedded diagonally in $E_\infty^\times$,  $dt^\times$ (resp.~$da^\times$) is our standard Haar measure on $\RR^\times$ (resp.~$(\A_E^\times)^1$), $db$ is a translation invariant measure on the $1$-dimensional free $\A_F$-module $\{x \in \A_E:\mathrm{tr}_{E/F'}(x)=0\}$, and $dk_H$ is a Haar measure on $K \cap H(\A_F)$.
As above, let $U \leq \GL_2$ be the unipotent radical of the Borel subgroup of upper-triangular matrices.
Replacing $h$ by $zuhk$ for $(z,u,k) \in Z_{G}(\A_F) \times \mathrm{Res}_{\OO_E/\OO_F}U(\A_F) \times K$ multiplies $\W_{h}(g)$ by a scalar of norm $1$ independent of $g$ and hence the above is equal to  
\begin{align*}
\int_{G(\A_F)^1}\left|X*\int_{\left(\begin{smallmatrix} ta& t^{-1}ab\\ & t^{-1}{}^{\iota}a^{-1} \end{smallmatrix} \right)k_H \in \mathfrak{F}} \sum_{y \in E^\times}\Psi*\W_{\left(\begin{smallmatrix} yt^2a\,{}^{\iota}a& \\ & 1 \end{smallmatrix} \right)}dt^\times da^\times dbdk_H*Y(g)\right|dg
\end{align*}
By definition, $\W_{\left(\begin{smallmatrix} yt^2a\,{}^{\iota}a& \\ & 1 \end{smallmatrix} \right)}$ vanishes identically if $(ya\,{}^{\iota}a)^\infty \not \in \widehat{\mathcal{D}}^{-1}_E$.  
Thus, since the integral over $a$ is compactly supported, the sum over $y$ above can be truncated to the subset $y \in \mathfrak{a}^{-1}-0$ for some ideal $\mathfrak{a} \subseteq \OO_E$. 
Moreover, for any $\varepsilon>0$, one has
$$
\int_{G(\A_F^\infty)}\one_m(g)dg \ll_\varepsilon |m|^{-1-\varepsilon}
$$
for $m \in \widehat{\OO}_E \cap \A_E^\times$, and hence the above is bounded by a constant depending on $\varepsilon$ times
\begin{align*}
\int_{G(F_\infty)}\left|X*\int_{\left(\begin{smallmatrix} ta& t^{-1}ab\\ & t^{-1}{}^{\iota}a^{-1} \end{smallmatrix} \right)k_H \in \mathfrak{F}} \sum_{y \in \mathfrak{a}^{-1}-0}\Psi*\W_{\left(\begin{smallmatrix} yt^2a\,{}^{\iota}a& \\ & 1 \end{smallmatrix} \right)_\infty}dt^\times |ya\,{}^{\iota}a|_\infty^{\varepsilon}da^\times dbdk_H*Y(g)\right|dg
\end{align*}
Notice also that the integral over $t$ can be truncated to the region $t>T$ for some $T>0$ by our choice of $\mathfrak{F}$ (see Lemma \ref{lem-F}).  Thus the integral above is bounded by Lemma \ref{lem-in-C}.  
\end{proof}

Finally, we prove a proposition which motivated our definition of $V$ in the first place.  Recall that if $\Pi$ is a cuspidal automorphic representation of $G(\A_F)^1$ distinguished by $H(\A_{F})^1=H(\A_F)$ then it is the base change of a cuspidal automorphic representation $\Pi'$ of $\GL_2(\A_{F'})^1$.  Moreover, in this case there is a certain $H(\A_{F})$-invariant linear functional on the space of Whittaker functions of $\Pi_S$ with respect to $\psi_{S}$ denoted in \cite[\S 3.3 and Definition 10.1]{FLO} by 
\begin{align} \label{a}
\alpha^{\Pi'_S}:=\prod_{v'|S} \alpha^{\Pi_{v'}'}_{\left(\begin{smallmatrix} & 1 \\ 1 & \end{smallmatrix} \right)}
\end{align}
where the product is over places $v'$ of $F'$ dividing $S$.    Since $H$ is quasi-split, these linear functionals $\alpha^{\Pi'_S}$ are nonzero by \cite[Corollary 10.3 and Remark 10.4]{FLO}.  We note that these distributions depend on the choice of measure on $H(F) \backslash H(\A_F)$ (see loc.~cit.).

\begin{prop} \label{prop-FLO} If $\Pi$ is a 
cuspidal automorphic representation of $G(\A_F)^1$, spherical at all places and $f \in C_c^\infty(G(\A_F)^1//K^S)$ then
$$
\mathrm{tr}\,\Pi(f*V)=\begin{cases}\sqrt{d}_E2\mathrm{tr}\,\Pi(f*\Psi)\alpha^{\Pi'_S}(W_{0S})L(1,\Pi' \times \Pi'^{\vee} \times \eta^S) &\textrm{ if }\Pi \cong {}^{\iota} \Pi\\
0& \textrm{ otherwise.}\end{cases}
$$
Here $\Pi'$ is a cuspidal automorphic representation of $G'(\A_F)^1=\GL_2(\A_{F'})^1$ whose base change to $G(\A_F)^1$ is $\Pi$ and $W_{0S} \in \mathcal{W}(\Pi_S,\psi_S)$ is the Whittaker newform.
\end{prop}
\begin{proof}
By construction the operator $\Pi(V)$ projects the space of $\Pi$ onto the line spanned by the Whittaker newform $\varphi$ of $\Pi$ as does $\Pi(\Psi)$.   Thus
\begin{align*}
\mathrm{tr}\,\Pi(f*V)&=\int_{ G(F) \backslash G(\A_F)^1} \frac{(\Pi(f*\int_{\mathfrak{F}}  \sum_{y \in E^\times}\Psi*\W_{\left(\begin{smallmatrix} y& \\ & 1 \end{smallmatrix} \right) h}dh)\varphi)(g)\overline{\varphi}(g)}{||\varphi||^2_2}dg\\
&=\int_{\mathfrak{F}} \sum_{y \in E^\times} W^{\varphi }\left(\left(\begin{smallmatrix} y& \\ & 1 \end{smallmatrix} \right) h\right)dh\int_{ G(F) \backslash G(\A_F)^1}\frac{(\Pi(f*\Psi)\varphi)(g)\overline{\varphi}(g)}{||\varphi||_2^2}dg\\
&=\mathrm{tr}\,\Pi(f*\Psi)\int_{\mathfrak{F}} \sum_{y \in E^\times} W^{\varphi }\left(\left(\begin{smallmatrix} y& \\ & 1 \end{smallmatrix} \right) h\right)dh.
\end{align*}

Recognizing the Whittaker expansion in the expression above (compare \eqref{Wh-exp}) we obtain
\begin{align*}
&\int_{\mathfrak{F}} \sum_{y \in E^\times} W^{\varphi }\left(\left(\begin{smallmatrix} y& \\ & 1 \end{smallmatrix} \right) h\right)dh
=\sqrt{d}_E\int_{H(F) \backslash H(\A_F)}  \varphi (h)dh.
\end{align*}
The proposition now follows from \cite[Theorem 10.2]{FLO}.

\end{proof}

\subsection{Proof of Theorem \ref{thm-main}} \label{ssec-proof}

By Theorem \ref{thm-tf} we have that
\begin{align} \label{first-sum}
TF(f^{G_0}*b_{E/F_0}(V))=\mathrm{tr}\,R_{\mathrm{cusp}}(f^{G_0}*b_{E/F_0}(V)) &=\sum_{\pi_0} \mathrm{tr}\,\pi_0(f^{G_0}*b_{E/F_0}(V))\\
&=\sum_{\pi_0} \mathrm{tr}\,\pi_0(f^{G_0})\mathrm{tr}\,\pi_0(b_{E/F_0}(V)) \nonumber
\end{align}
where the (absolutely convergent) sum is over all cuspidal automorphic representations of $G_0(\A_F)^1$.  
Recall that each $\pi_{0}$ admits a unique cuspidal base change $\pi_{0E}$ to $G(\A_F)$, and $\pi_{0E} \cong \sigma_{0E}$ if and only if $\pi_0 \cong \sigma_0 \otimes \chi$ for some character $\chi$ of $F_0 \backslash \A_{F_0}^\times$ trivial on  $\N_{E/F_0}(\A_{F_0}^\times)$.  This was proven by Langlands \cite{LanglBC} in the prime degree case, and for the deduction of the general case, which is easy since the extension $E/F'$ has odd degree, see \cite{LapidRog}. 

Thus, by Lemma \ref{lem-C-trans}, the sum \eqref{first-sum} is equal to 
\begin{align} \label{just-bc}
\sum_{\pi_0} \mathrm{tr}\,\pi_{0E}(f)\mathrm{tr}\,\pi_{0E}(V)=[E:F_0]\sum_{\Pi:\Pi \cong \Pi^{\tau}} \mathrm{tr}\,\Pi(f)\mathrm{tr}\,\Pi(V)
\end{align} 
where the latter sum is over all cuspidal automorphic representations of $G(\A_F)^1$ invariant under $\tau$.  The theorem now follows from Proposition \ref{prop-FLO}.
\qed

\section*{Acknowledgments}

We thank R.~Langlands, B.~C.~Ng\^o, and P.~Sarnak for their generous help and encouragement over the past few years, and E.~Lapid for pointing out a mistake in an earlier version of this paper that lead to the simpler approach exposed here.  The first author also thanks H.~Hahn for her constant encouragement and help with proofreading.



\end{document}